\documentclass[12pt,reqno]{amsart}

\usepackage{times}
\usepackage{graphicx}
\usepackage{romannum}

\usepackage{amsmath,amsthm,amsfonts,amscd,amssymb}
\usepackage{pst-node}
\usepackage{pst-plot}

\usepackage{amssymb}
\usepackage{MnSymbol}

\usepackage{appendix}
\usepackage[all]{xy}

\usepackage{tikz, braids}
\usetikzlibrary{decorations.pathreplacing, arrows}

\newtheorem{theorem}{Theorem}[section]
\newtheorem{corollary}[theorem]{Corollary}

\newtheorem{lemma}[theorem]{Lemma}

\newtheorem{proposition}[theorem]{Proposition}
\newtheorem{example}[theorem]{Example}

\numberwithin{theorem}{section}
\numberwithin{equation}{section}

\newcommand{\norm}[1]{\left\Vert#1\right\Vert}

\newcommand{\la}{\langle}
\newcommand{\ra}{\rangle}

\newcommand{\Comp}{\mathbb{C}}

\newcommand{\n}{\mathbb{N}}

\global\long\def\tp{\mathop{\xymatrix{*+<.7ex>[o][F-]{\scriptstyle \top}}
 } }

\begin{document}

\title[Strong Haagerup inequalities on non-Kac $O_F^+$]{Strong Haagerup inequalities on non-Kac free orthogonal quantum groups}

\author {Sang-Gyun Youn}


\address{Sang-Gyun Youn, 
Department of Mathematics Education, Seoul National University, 
Gwanak-ro 1, Gwanak-gu, Seoul, Republic of Korea 08826}
\email{s.youn@snu.ac.kr}


\maketitle

\begin{abstract}
We present natural analogues of strong Haagerup inequalities on non-Kac free orthogonal quantum groups $O_F^+$ in which $L^p$-analytic problems are harder due to their non-tracial nature. Furthermore, we prove optimality of the inequalities, and apply the obtained results to compute the optimal time for  ultracontractivity of the heat semigroup and to distinguish the complex interpolation space $L^p(O_F^+)$ and the real interpolation space $L^{p,p}(O_F^+)$.
\end{abstract}

\section{Introduction}

Let $\mathbb{F}_n$ be the non-abelian free group with generators $g_1,g_2,\cdots,g_n$ and let $C_r^*(\mathbb{F}_n)$ be the associated reduced group $C^*$-algebra.  There exists a natural word length function $|\cdot|: \mathbb{F}_n\rightarrow \n_0=\left \{0\right\}\cup \n$ with respect to the generating set $E_1=\left \{g_1,g_1^{-1},\cdots, g_n,g_n^{-1}\right\}$. In \cite{Ha79}, Haagerup proved the following result
\begin{equation}\label{eq-Haagerup}
\norm{\sum_{g\in E_k}f(g) \lambda_g}_{C_r^*(\mathbb{F}_n)}\leq (k+1)\left ( \sum_{g\in E_k}|f(g)|^2\right )^{\frac{1}{2}}
\end{equation}
for any $f:E_k\rightarrow \Comp$ where $E_k=\left \{x\in \mathbb{F}_n: |x|= k\right\}$, and applied \eqref{eq-Haagerup} to prove {\it metric approximation property} of $C_r^*(\mathbb{F}_n)$. The inequality \eqref{eq-Haagerup} is called the {\it Haagerup inequality} and has been studied for various other groups \cite{Jo90,Ha88,La00,Ch05,CoYo19} under the name {\it property RD} whose earliest definition was given by Connes \cite{Co94} in view of non-commutative geometry. See \cite{CoMo90,La02,Ch17} for more results and monumental applications of property RD.

\vspace{4pt}

A strengthened form of \eqref{eq-Haagerup}, namely {\it strong Haagerup inequality}, was introduced in \cite{KeSp07} for $*$-free R-diagonal families. In particular, let $E_k^+$ be the subset of elements in $E_k$ generated only by $g_1,g_2,\cdots,g_k$ without their inverses. Then we have
\begin{equation}\label{eq-strong-Haagerup}
\norm{\sum_{g\in E_k^+}f(g) \lambda_g}_{C_r^*(\mathbb{F}_n)}\lesssim \sqrt{k+1}\left ( \sum_{g\in E_k^+}|f(g)|^2\right )^{\frac{1}{2}}
\end{equation}
for any $f:E_k^+\rightarrow \Comp$. Here, $X\lesssim Y$ means that $X\leq KY$ for some universal constant $K$ which is independent of $X$ and $Y$. The non-self-adjoint subalgebra generated by $\left \{\lambda_g: g\in E_k^+\right\}$ can be considered a general ``{\it holomorphic}'' setting \cite{KeSp07} and the striking feature of \eqref{eq-strong-Haagerup} is the improvement of the constants from $k+1$ to $\sqrt{k+1}$. Subsequently, \eqref{eq-strong-Haagerup} was studied within the category of operator spaces in \cite{dlSa09}, and the results of \cite{KeSp07} were generalized in \cite{Br12c} with sharp inequalities on Kac free unitary quantum groups $U_N^+$.

\vspace{4pt}

The notion of property RD was extended to discrete quantum groups and was proved for Kac orthogonal free quantum groups in \cite{Ve07}. More precisely, \cite{Ve07} proved the following result
\begin{equation}
\norm{f}_{C_r(O_N^+)}\lesssim (k+1)\norm{f}_{L^2(O_N^+)}
\end{equation}
for any {\it homogenous} polynomial $f\in C_r(O_N^+)$ of degree $k$, where $\norm{\cdot}_{C_r(O_N^+)}$ is the reduced $C^*$-norm and $\norm{\cdot }_{L^p(O_N^+)}$ is the noncommutative $L^p$-norm.  This {\it quantum property RD} has made various operator algebraic applications \cite{Br12,Br14,VaVe07,Ve12,Yo18a} and interesting connections with quantum entanglement \cite{BrCo18,BCLY20}. 

\vspace{4pt}

In the case of non-Kac quantum group examples, two natural definitions of property RD have been introduced in \cite{Ve07} and \cite{BVZ15} respectively, and it has turned out in \cite{BVY21} that the property RD from either point of view is not satisfied for any non-amenable non-Kac orthogonal free quantum group. Nevertheless, in \cite{BVY21}, a weakened property RD
\begin{equation}\label{eq03}
\norm{f}_{C_r(O_F^+)}\lesssim (k+1)\norm{F}^{2k}\norm{f}_{L^2(O_F^+)}
\end{equation}
was proved for all homogeneous polynomials $f\in C_r(O_F^+)$ of degree $k$. 

\vspace{4pt}

We revisit \cite{BVY21} and optimize their key arguments for non-Kac $O_F^+$ cases to show that \eqref{eq03} can be sharpened to 
\begin{equation}\label{eq07}
\norm{f}_{C_r(O_F^+)}\lesssim \norm{F}^{2k}\norm{f}_{L^2(O_F^+)}
\end{equation}
in Theorem \ref{thm1}. Moreover, the main point of Section \ref{sec:rd} is that the function $\norm{F}^{2k}$ can be replaced by other slower growing functions up to some natural choices of generators as parallels to the strong Haagerup inequalities (Theorem \ref{thm1}, Corollary \ref{cor31}) and that the inequalities on {\it analytic} polynomials with single generators are indeed optimal even for non-homogeneous cases  (Theorem \ref{thm31}). Let us exhibit an extremal inequality for a canonical matrix $F=\displaystyle \sum_{i=1}^N \lambda_i e_{i,N+1-i}\in M_N(\mathbb{R})$ such that $F^2=\pm \text{Id}_N$ and $|\lambda_i|$ is monotone increasing with 
\begin{equation}
0<|\lambda_1|\leq |\lambda_2|\leq \cdots \leq |\lambda_n|<1<|\lambda_{N-n+1}|\leq \cdots \leq |\lambda_N|.
\end{equation}
In this case $C_r(O_F^+)$ have $N^2$ generators $u_{ij}$ ($1\leq i,j\leq N$) and one of our key results is
\begin{equation}\label{eq07.5}
\norm{f}_{C_r^*(O_F^+)}\approx \norm{f^*}_{L^2(O_F^+)}
\end{equation} 
for any homogeneous {\it analytic} polynomials $f=\displaystyle \sum_{1\leq i_1,\cdots,j_k\leq n} x_{i_1\cdots j_k}u_{i_1j_1}\cdots u_{i_kj_k}$ of degree $k$. More generally, non-commutative $L^p$-norms are estimated by\small
\begin{align}\label{eq07.6}
\norm{f}_{L^p(O_F^+)}\approx \left ( \frac{2}{N_q+\sqrt{N_q^2-4}} \right )^{\frac{k}{2}} \left ( \sum_{1\leq i_1,\cdots,j_k\leq n} | x_{i_1\cdots j_k}|^2   \left [ \lambda_{i_1}\cdots \lambda_{i_k} \right ]^{\frac{4}{p}} \left [ \lambda_{j_1}\cdots \lambda_{j_k} \right ]^{\frac{4}{p}-2} \right )^{\frac{1}{2}}
\end{align}
\normalsize for any $1\leq p\leq \infty$ and homogeneous analytic polynomials of degree $k$ given by $f=\displaystyle \sum_{1\leq i_1,\cdots,j_k\leq n} x_{i_1\cdots j_k}u_{i_1j_1}\cdots u_{i_kj_k}$. Here, $N_q=\text{Tr}(F^*F)$ and the constants in \eqref{eq07.6} depend only on the matrix $F$. More general results are established in Corollary \ref{cor31} and Example \ref{ex31}. This phenomenon does not appear on Kac free orthogonal quantum groups $O_N^+$.

\vspace{4pt}

The main result of Section \ref{sec:ultracontractivity} is the characterization of all $L^2$-$L^{\infty}$ {\it central multipliers} $T_{\varphi}$, identified with functions $\varphi:\n_0\rightarrow (0,\infty)$ such that $T_{\varphi}(f_k)=\varphi(k)f_k$ for any homogeneous polynomials $f_k$ of degree $k$. In Theorem \ref{thm-ultra1}, we prove that a central multiplier $T_{\varphi}$ satisfies
\begin{equation}
\norm{T_{\varphi}(f)}_{C_r(O_F^+)}\lesssim \norm{f}_{L^2(O_F^+)}
\end{equation}
for all polynomials $f$ on $O_F^+$ if and only if $\displaystyle \sum_{k=0}^{\infty}\varphi(k)^2\norm{F}^{4k}<\infty $. A remarkable outcome from this characterization is the accurate optimal time $t_F$ for {\it ultracontractivity} of the heat semigroup $(\Phi_t)_{t>0}$ satisfying that $\Phi_t:L^2(O_F^+)\rightarrow C_r(O_F^+)$ is bounded if and only if $t>t_F$. It was noted in \cite{BVY21} that  the optimal time $t_F$ for non-Kac case is non-trivial with the following bounds
\begin{equation}\label{eq04}
2(N_q-2)\log\norm{F}\leq t_F\leq 2 N_q\log\norm{F}.
\end{equation}
We apply the obtained results from Section \ref{sec:rd} with a detailed analysis of the heat semigroup (Lemma \ref{lem40}) to prove $t_F=2\sqrt{N_q^2-4}\log \norm{F}$ for any free orthogonal quantum group $O_F^+$ in Corollary \ref{cor:ultra2}.

\vspace{4pt}

In Section \ref{sec:interpolation}, we exhibit distinct differences between the complex interpolation space $L^p(O_F^+)=(L^{\infty}(O_F^+),L^1(O_F^+))_{\frac{1}{p}}$ and the real interpolation space $L^{p,p}(O_F^+)=(L^{\infty}(O_F^+),L^1(O_F^+))_{\frac{1}{p},p}$ for non-Kac $O_F^+$. Although these two interpolation methods are compatible in the tracial setting, it is no longer true in the non-tracial setting as pointed out in \cite{PiXu03}. As an analogous approach within the framework of compact quantum groups, we apply the results from Section \ref{sec:rd} to prove that the complex interpolation space $L^p(O_F^+)$ and the real interpolation space $L^{p,p}(O_F^+)$ have equivalent norms for some $1<p\neq 2<\infty$ if and only if $O_F^+$ is of Kac type.

\bigskip

\noindent {\bf Acknowledgments.} This research was supported by National Research Foundation of Korea (NRF) grant funded by the Korea government (MSIT) (No. 2020R1C1C1A01009681).

\section{Preliminaries}\label{sec:pre}

\subsection{Free orthogonal quantum group}

One of the most important examples of {\it compact quantum groups} is the so-called {\it free orthogonal quantum group} $O_F^+$, which has been introduced in \cite{Wa95,VaWa96}. See \cite{Wo87a,Wo87b,Ti08} and \cite{KuVa00,KuVa03} for more details of (locally) compact quantum groups.

For any invertible $N\times N$ matrix $F\in GL_N(\Comp)$ such that $\overline{F}F=\pm \text{Id}_N$, the free orthogonal quantum group $O_F^+$ is given by 
\begin{enumerate}
\item a universal unital $C^*$-algebra $C(O_F^+)$ generated by $N^2$ elements $u_{ij}$ ($1\leq i,j\leq N$) satisfying that $u=\displaystyle \sum_{i,j=1}^N e_{ij}\otimes u_{ij}\in M_N(\Comp)\otimes C(O_F^+)$ is unitary and $u=(F\otimes 1)u^c(F^{-1}\otimes 1)$ where $u^c=\displaystyle \sum_{i,j=1}^N e_{ij}\otimes u_{ij}^*$.
\item a unital $*$-homomorphism $\Delta:C(O_F^+)\rightarrow C(O_F^+)\otimes_{\text{min}}C(O_F^+)$ satisfying $\Delta(u_{ij})=\displaystyle \sum_{k=1}^N u_{ik}\otimes u_{kj}$ for all $1\leq i,j\leq N$.
\end{enumerate}

The unital $*$-algebra generated by $u_{ij}$ $(1\leq i,j\leq N)$ is a dense subspace of $C(O_F^+)$, which we denote by $\text{Pol}(O_F^+)$. We call any element in $\text{Pol}(O_F^+)$ a {\it polynomial} on $O_F^+$. There exists  a unique state $h:C(O_F^+)\rightarrow \Comp$ such that
\begin{equation}
(\text{id}\otimes h)\circ \Delta = h(\cdot)1 = (h\otimes \text{id})\circ \Delta.
\end{equation}
The state $h$ is faithful on $\text{Pol}(O_F^+)$ and we call $h$ the {\it Haar state} on $O_F^+$. The completion of $\text{Pol}(O_F^+)$ with respect to the inner product 
\begin{equation}
\la f,g\ra=h(g^*f),~f,g\in \text{Pol}(O_F^+),
\end{equation}
is denoted by $L^2(O_F^+)$. The associated GNS representation $\pi:C(O_F^+)\rightarrow B(L^2(O_F^+))$ is given by
\begin{equation}
\pi(a)b=ab,~a,b\in \text{Pol}(O_F^+)
\end{equation}
and the image $\pi(C(O_F^+))$ is denoted by $C_r(O_F^+)$. We denote the associated von Neumann algebra $C_r(O_F^+)''\subseteq B(L^2(O_F^+))$ by $L^{\infty}(O_F^+)$. The Haar state $h$ naturally extends to a normal faithful state $h$ on $L^{\infty}(O_F^+)$.

It is known from \cite{BiDeVa06,BrKi16} that free orthogonal quantum groups $O_{F_1}^+$ and $O_{F_2}^+$ with $\overline{F_1}F_1=\pm \text{Id}_N=\overline{F_2}F_2$ are isomorphic if and only if $F_2=wF_1w^t$ for some unitary matrix $w$. Moreover, with slight modifications on the choices of unitaries, we obtain the following canonical form of $F\in GL_N(\Comp)$ satisfying $\overline{F}F=\pm \text{Id}_N$:

\begin{lemma}\label{lem:canonical}
For any $F\in GL_N(\Comp)\setminus \mathcal{U}(N)$ such that $\overline{F}F=\pm \text{Id}_N$ there exists a unitary $w\in \mathcal{U}(N)$ such that \begin{equation}
wFw^t=\displaystyle \sum_{i=1}^N \lambda_i e_{i,N+1-i}\in M_N(\mathbb{R})
\end{equation} 
such that $(wFw^t)^2=\pm \text{Id}_N$, $0<|\lambda_1|\leq  \cdots \leq |\lambda_n|<1<|\lambda_{N-n+1}|\leq \cdots \leq |\lambda_N|$ and $|\lambda_j|=1$ for all $n<j\leq N-n$.
\end{lemma}

\begin{proof}
Let us divide the two cases (1) $\overline{F}F=\text{Id}_N$ and (2) $\overline{F}F=-\text{Id}_N$. A real symmetric matrix $J_n=\displaystyle \sum_{i=1}^n e_{i, n+1-i}\in M_n(\mathbb{R})$ has eigenvalues $\pm 1$, so there exists an orthogonal matrix $G\in \mathcal{O}(n)$ such that 
\begin{align}
J_n&=G\cdot D(1, \cdots, 1,-1,\cdots,-1)\cdot G^{t}\\
=&[G\cdot D(1,\cdots, 1,i,\cdots,i)]\cdot [G\cdot D(1,\cdots, 1,i,\cdots,i)]^t.
\end{align} 
where $D(x_1,\cdots, x_m)$ is the $n\times n$ diagonal matrix with entries $x_1,\cdots,x_m$. Let us denote by $V_n$ the above $G\cdot D(1,\cdots,1,i,\cdots,i)$.
\begin{enumerate}
\item If $\overline{F}F=\text{Id}_N$, then by \cite{BiDeVa06,BrKi16} there exists $w_0\in \mathcal{U}(N)$ such that  
\begin{equation}
w_0Fw_0^t=\displaystyle \left [ \begin{array}{ccc}
0&D(\lambda_1,\cdots,\lambda_n)&0\\
D(\lambda_1^{-1},\cdots,\lambda_n^{-1})&0&0\\
0&0&\text{Id}_{N-2n}\\
\end{array}\right ]
\end{equation} 
with $0<\lambda_1\leq \cdots \leq \lambda_n<1$. Let us take $w_1=\displaystyle \left [ \begin{array}{ccc}
\text{Id}_n&0&0\\
0&0&V_{N-2n}\\
0&J_n&0
\end{array}\right ]$. Then for $w=w_1w_0$ we have
\begin{align}
wFw^t=&\left [ \begin{array}{ccc}
0&0&D(\lambda_1,\cdots,\lambda_n)J_n\\
0&V_{N-2n}V_{N-2n}^t&0\\
J_nD(\lambda_1^{-1},\cdots,\lambda_n^{-1})&0&0
\end{array}\right ]\\
&=\left [ \begin{array}{ccc}
0&0&\sum_{i=1}^n \lambda_i e_{i,n+1-i}\\
0&J_{N-2n}&0\\
\sum_{i=1}^n \lambda_{n+1-i}^{-1} e_{i,n+1-i}&0&0
\end{array}\right ],
\end{align}
which is the desired form.
\item If $\overline{F}F=-\text{Id}_N$, then $N$ has to be an even number and, by \cite{BiDeVa06,BrKi16}, there exists $w_0\in \mathcal{U}(N)$ such that  
\begin{equation}
w_0Fw_0^t=\displaystyle \left [ \begin{array}{ccc}
0&D(\lambda_1,\cdots,\lambda_n,1,\cdots,1)\\
-D(\lambda_1^{-1},\cdots,\lambda_n^{-1},1,\cdots,1)&0
\end{array}\right ]
\end{equation} 
with $0<\lambda_1\leq \lambda_2\leq \cdots \leq \lambda_n<1$. Then, using $w_1=\left [\begin{array}{cc} \text{Id}_{\frac{N}{2}}&0\\ 0&J_{\frac{N}{2}} \end{array} \right ]$, we reach the conclusion similarly.
\end{enumerate}

\end{proof}

From now on, let us suppose that $F$ is of an anti-diagonal form as in Lemma \ref{lem:canonical}, which we call a {\it canonical $F$-matrix} (differently from \cite{BiDeVa06,BrKi16}). Then $F^*F$ and all submatrices of $(F^*F)^{\otimes k}$ are diagonal matrices, and we have
\begin{equation}\label{eq25}
u_{ij}= \lambda_{i} \lambda_{j}^{-1} u_{N+1-i,N+1-j}^*,~1\leq i,j\leq N,
\end{equation}
from the defining relation $u=(F\otimes 1)u^c(F^{-1}\otimes 1)$.

\subsection{Fourier series and Schur orthogonality}

A finite dimensional {\it unitary representation} of $O_F^+$ is given by a unitary $v=\displaystyle \sum_{i,j=1}^{n_v}e_{ij}\otimes v_{ij}\in M_{n_v}(\Comp)\otimes \text{Pol}(O_F^+)$ such that $\Delta(v_{ij})=\displaystyle \sum_{k=1}^{n_v}v_{ik}\otimes v_{kj}$ for all $1\leq i,j\leq n_v$. Furthermore, $v$ is called {\it irreducible} if 
\begin{equation}
\left \{A\in M_{n_v}(\Comp): (A\otimes 1)v= v(A\otimes 1) \right\}=\Comp\cdot \text{Id}_{n_v}.
\end{equation}
It was shown in \cite{Ba96} that all irreducible unitary representations of $O_F^+$ are labeled by $u^{(0)},u^{(1)}, u^{(2)},\cdots$ up to unitary equivalence. In particular, $u^{(0)}=1$ is the 1-dimensional trivial representation and the {\it fundamental representation} $u^{(1)}$ is given by $\displaystyle u=\sum_{i,j=1}^N e_{ij}\otimes u_{ij}$ where $u_{ij}$'s are the generators of $C(O_F^+)$. The $k$-fold tensor representation of $u$ is \small
\begin{equation}
u^{\tiny \tp \normalsize k}=u\tp u\tp \cdots \tp u= \sum_{\substack{1\leq i_1,\cdots, i_k\leq N\\ 1\leq j_1,\cdots,j_k\leq N}}e_{i_1j_1}\otimes e_{i_2j_2}\otimes \cdots \otimes e_{i_kj_k}\otimes u_{i_1j_1} u_{i_2j_2} \cdots  u_{i_kj_k}.
\end{equation}
\normalsize

We define the subspace $H_k\subseteq (\Comp^N)^{\otimes k}$ as the space of all $\xi\in (\Comp^N)^{\otimes k}$ satisfying
\begin{equation}
(\text{id}_N^{\otimes i}\otimes T^*\otimes \text{id}_N^{\otimes (N-2-i)})\xi=0
\end{equation}
for all $i=0,1,\cdots,N-2$, where $T=\displaystyle \sum_{j=1}^N  e_j\otimes F e_j$. Then the $k$-th irreducible unitary representation $u^{(k)}$ is given by  $(\iota_k^* \otimes 1)u^{\tiny \tp \normalsize k}(\iota_k\otimes 1)\in B(H_k)\otimes \text{Pol}(O_F^+)$ where $\iota_k: H_k\hookrightarrow (\Comp^N)^{\otimes k}$ is the isometric embedding. We denote by $n_k=\text{dim}(H_k)$ the {\it classical dimension} of $u^{(k)}$ and write $u^{(k)}=\displaystyle \sum_{i,j=1}^{n_k}e_{ij}\otimes u^{(k)}_{ij}$. The space spanned by $u^{(k)}_{ij}$ for fixed $k$ is denoted by
\begin{equation}
\text{Pol}_k(O_F^+)=\text{span}\left \{ u^{(k)}_{ij}: 1\leq i,j\leq n_k \right \}
\end{equation}
and call any element in $\text{Pol}_k(O_F^+)$ a {\it homogeneous polynomial} of degree $k$. Furthermore, any polynomial is spanned by homogeneous polynomials, i.e. we have $\displaystyle \text{Pol}(O_F^+)=\bigoplus_{k\in \n_0} \text{Pol}_k(O_F^+)=\text{span}\left \{u^{(k)}_{ij}:  ~k\in \n_0, ~1\leq i,j\leq n_k \right \}$.

Schur's orthogonality relations (with the above choices of ONB's) tell us that there exist invertible positive diagonal matrices $Q(k)\in M_{n_k}(\Comp)$ such that
\begin{align}
\label {eq22}h\left ((u^{(k_1)}_{i_1j_1})^*u^{(k_2)}_{i_2j_2}\right )&=\delta_{k_1,k_2}\delta_{i_1,i_2}\delta_{j_1,j_2}\frac{Q(k)_{i_1i_1}^{-1}}{\text{Tr}(Q(k))} ,\\
h\left (u^{(k_1)}_{i_1j_1}(u^{(k_2)}_{i_2j_2})^*\right )&=\delta_{k_1,k_2}\delta_{i_1,i_2}\delta_{j_1,j_2}\frac{Q(k)_{j_1j_1}}{\text{Tr}(Q(k))}.
\end{align}
We call $d_k=\text{Tr}(Q(k))=\text{Tr}(Q(k)^{-1})$ the {\it quantum dimension} of $u^{(k)}$ and say that $O_F^+$ is of {\it Kac type} if $d_k=n_k$, i.e. $Q(k)=\text{Id}_{n_k}$ for all $k\in \n_0$. In particular, $Q(1)=F^*F$ and we denote by $N_q$ the associated quantum dimension $d_1=\text{Tr}(Q(1))=\text{Tr}(F^*F)$.

In particular, \eqref{eq22} implies that a canonical orthonormal basis of $L^2(O_F^+)$ is given by 
\begin{equation}
\left \{\sqrt{d_k Q(k)_{ii}}u^{(k)}_{ij}:  ~k\in \n_0, ~1\leq i,j\leq n_k \right \}.
\end{equation}
Hence, for each $f\in \text{Pol}(O_F^+)$, there exist uniquely determined matrices $A(k)\in M_{n_k}(\Comp)$ such that 
\begin{equation}\label{eq251}
f=\displaystyle \sum_{k\geq 0}\sum_{i,j=1}^{n_k} d_k (A(k)Q(k))_{ij}u^{(k)}_{ji}
\end{equation}
The matrices $A(k)$ are denoted by $\widehat{f}(k)$ and called the {\it Fourier coefficients} of $f\in \text{Pol}(O_F^+)$. We call \eqref{eq251} the {\it Fourier series} of $f\in \text{Pol}(O_F^+)$ and the following {\it Plancherel identity}
\begin{equation}\label{eq20}
\norm{f}_{L^2(O_F^+)}=h(f^*f)^{\frac{1}{2}}=\left (\sum_{k\geq 0}d_k \text{Tr}(\widehat{f}(k)^*\widehat{f}(k)Q(k)) \right )^{\frac{1}{2}}
\end{equation}
holds for any $f\in \text{Pol}(O_F^+)$. We refer the reader to \cite{Ti08} for more general descriptions of Schur orthogonality relations on compact quantum groups.

\subsection{The complex interpolation and real interpolation methods}

Let us gather some basics of the complex and the real interpolation methods on Banach spaces from \cite{BeLo76}. Note that these methods have natural analogues in the category of operator spaces \cite{Pi96,Xu96}. 

A pair of Banach spaces $(X_0,X_1)$ is called {\it compatible} if there exists a topological vector space $X$ with continuous embeddings $j_0:X_0\rightarrow X$ and $j_1:X_1\rightarrow X$. Let us write $j_0(X_0)+j_1(X_1)$ simply as $X_0+X_1$. For any $v\in X_0+X_1$ we define
\begin{equation}
\norm{v}_{X_0+X_1}=\inf \left \{\norm{x_0}_{X_0}+\norm{x_1}_{X_1}: x=x_0+x_1, x_0\in X_0, x_1\in X_1 \right \}.
\end{equation}

For the complex method, let us take $S=\left \{z\in \Comp: 0< \text{Re}(z)< 1\right\}$ and denote by $\mathcal{F}(X_0,X_1)$ the set of all continuous bounded functions $f:\overline{S}\rightarrow X_0+X_1$ such that 
\begin{enumerate}
\item $f|_S:S\rightarrow X_0+X_1$ is analytic and
\item the restrictions $f|_{\partial_0}:\partial_0\rightarrow X_0$ and $f|_{\partial_1}:\partial_1\rightarrow X_1$ are bounded.
\end{enumerate} 
For any $0<\theta<1$ the complex interpolation space 
\begin{equation}
X_{\theta}=(X_0,X_1)_{\theta}=\left \{ f(\theta): f\in \mathcal{F}(X_0,X_1)\right\} \subset X_0+ X_1
\end{equation} 
is a Banach space with respect to 
\begin{equation}
\norm{x}=\inf \left \{ \max \left \{\sup_{z\in \partial_0}\norm{f(z)}_{X_0}, \sup_{z\in \partial_1}\norm{f(z)}_{X_1}\right\}:f\in \mathcal{F}(X_0,X_1),f(\theta)=x\right\}.
\end{equation}

For the real interpolation, let us explain some basics of the K-method and the J-method. For any $t>0$ the K-functional of $x\in X_0+ X_1$ is defined by
\begin{align}
K(x,t)&=\inf\left \{ \norm{x_0}_{X_0}+t\norm{x_1}_{X_1}: x=x_0+x_1,x_0\in X_0,x_1\in X_1 \right\}.
\end{align}
We define
\begin{align}
\norm{x}_{\theta,q,K}=\left ( \int_0^{\infty} (t^{-\theta}K(x,t))^{q}\frac{dt}{t} \right )^{\frac{1}{q}},~x\in X_0+X_1
\end{align}
for any $0<\theta<1$ and $1\leq q<\infty$, and 
\begin{align}
\norm{x}_{\theta,\infty,K}= \sup_{t>0}t^{-\theta}K(x,t),~0\leq \theta \leq 1.
\end{align}
The real interpolation space $(X_0,X_1)_{\theta,q,K}$ by the K-method is defined as the subspace of all elements $x$ in $X_0 + X_1$ such that $\norm{x}_{\theta,q,K}<\infty$. On the other side, for any $t>0$,  the J-functional of $x\in X_0\cap X_1$ is defined by
\begin{equation}
J(x,t)=\max \left \{ \norm{x}_{X_0},t \norm{x}_{X_1}\right\}
\end{equation}
and we define
\begin{equation}
\norm{x}_{\theta,q,J}=\inf \left ( \int_0^{\infty}(t^{-\theta}J(f(t),t))^q \frac{dt}{t} \right )^{\frac{1}{q}},~x\in X_0+X_1,
\end{equation}
where the infimum runs over all strongly measurable function $f:(0,\infty)\rightarrow X_0\cap X_1$ satisfying
\begin{equation}\label{eq21}
x=\int_0^{\infty} f(t)\frac{dt}{t}.
\end{equation}
Here the convergence in \eqref{eq21} is with respect to the K-functional $K(\cdot , 1)$. The real interpolation space $(X_0,X_1)_{\theta,q,J}$ by the $J$-method is defined as the space of all elements in $X_0+X_1$ for which $\norm{x}_{\theta,p,J}<\infty$. 

Since the real interpolation spaces $(X_0,X_1)_{\theta,q,K}$ and $(X_0,X_1)_{\theta,q,J}$ have equivalent norms for any $0<\theta<1$ and $1\leq q\leq \infty$ \cite[Theorem 2.8]{BeSh88}, let us denote by $(X_0,X_1)_{\theta,q}$ the common space. In the case of $X$-valued $L^p$-spaces with a Banach space $X$, it is known that 
\begin{align}
L^p(\mu;X)&=(L^{\infty}(\mu ;X),L^1(\mu;X))_{\frac{1}{p}}\\
=&(L^{\infty}(\mu ;X),L^1(\mu ;X))_{\frac{1}{p},p}=L^{p,p}(\mu;X)
\end{align}
where $L^{p,q}(\mu; X)$ is the $X$-valued Lorentz space.

In the case of free orthogonal quantum group $O_F^+$, we define $L^1(O_F^+)$ as the predual of the von Neumann algebra $L^{\infty}(O_F^+)$. Then $L^{\infty}(O_F^+)$ is naturally embedded into $L^1(O_F^+)$ as a dense subspace $\iota(L^{\infty}(O_F^+))$ by
\begin{equation}
\la \iota(x) ,y \ra_{L^1(O_F^+),L^{\infty}(O_F^+)}=h(yx),~x,y\in L^{\infty}(O_F^+).
\end{equation}
The non-commutative $L^p$-space $L^p(O_F^+)$ is defined as the complex interpolation space $(L^{\infty}(O_F^+),L^1(O_F^+))_{\frac{1}{p}}$ for any $1< p < \infty$, and the resulting norm structure of $(L^{\infty}(O_F^+),L^1(O_F^+))_{\frac{1}{p}}$ for $p=2$ is exactly same with the $L^2$-norm that we defined in \eqref{eq20}. The space of polynomials $\text{Pol}(O_F^+)$ is dense in $L^p(O_F^+)$ for any $1\leq p<\infty$.

\section{Strong Haagerup inequalities}\label{sec:rd}

Let $S(l,k)=\left \{|l-k|,|l-k|+2,\cdots, l+k\right\}$ and write $r=\displaystyle \frac{k+n-l}{2}$ for any $n\in S(l,k)$. We denote by $p_l$ the orthogonal projection from $L^2(O_F^+)$ onto $\text{Pol}_l(O_F^+)$. Let us rewrite some key arguments in \cite{BVY21} under our notations. Incorporating the proofs of Lemma 2.2 and Proposition 3.4 in \cite{BVY21}, we have
\begin{equation}\label{eq30}
\norm{p_l f p_n}_{B(L^2(O_F^+))}\leq C_q \sqrt{d_k} \left (\sum_{i,j=1}^{n_r}\text{Tr}\left (a_{ij}^*a_{ij}Q(k-r)\right )  Q(r)_{ii}Q(r)_{jj}^2 \right )^{\frac{1}{2}}
\end{equation}
for any $f\in \text{Pol}_k(O_F^+)$ and $n\in S(l,k)$, where the Fourier coefficient $\widehat{f}(k)$ is given by $\displaystyle \sum_{i,j=1}^{n_r}a_{ij}\otimes E_{ij}$ under the embedding $B(H_k)\hookrightarrow B(H_{k-r})\otimes B(H_r)$ and $E_{ij}=e_i e_j^*$ with respect to the canonical orthonormal basis $(e_i)_{i=1}^{n_r}$ of $H_r$. Here, the constant $C_q$ is given by
\begin{equation}
C_q= \frac{1}{1-q^2}\left ( \prod_{m=1}^{\infty}\frac{1}{1-q^{2m}}\right )^3
\end{equation}
where $q\in (0,1)$ is a solution of $x+x^{-1}=N_q=\text{Tr}(F^*F)$.

Let us denote by $\displaystyle M_f(l,n)=\max_{\substack{1\leq i,j\leq n_r\\ a_{ij}\neq 0}}\left \{Q(r)_{ii}^{\frac{1}{2}}Q(r)_{jj}^{\frac{1}{2}}\right\}\leq \norm{F}^{k+n-l}$. Then we have 
\begin{equation}\label{eq30.5}
\norm{p_l f p_n}_{B(L^2(O_F^+))}\leq C_q M_f(l,n) \norm{f}_{L^2(O_F^+)}
\end{equation}
since $\displaystyle d_k\text{Tr}\left (\widehat{f}(k)^*\widehat{f}(k)Q(k)\right )=d_k \sum_{i,j=1}^{n_r}\text{Tr}(a_{ij}^*a_{ij}Q(k-r))\cdot Q(r)_{jj}$. 

Recall that \eqref{eq30.5} was an important step in \cite{BVY21} to obtain
\begin{equation}\label{eq30.6}
\norm{f}_{C_r(O_F^+)}\lesssim (k+1)\norm{F}^{2k}\norm{f}_{L^2(O_F^+)}
\end{equation} 
for any homogeneous polynomials $f\in \text{Pol}_k(O_F^+)$. Now, let us sharpen \eqref{eq30.6} using some optimized arguments for non-Kac $O_F^+$ and the exponential growth of $\norm{F}^{2k}$.

\begin{theorem}\label{thm1}
Let $O_F^+$ be a non-Kac free orthogonal quantum group and $f\in \text{Pol}_k(O_F^+)$ such that $M_f(l,n)\leq R^{k+n-l}$ for some positive $R\neq 1$. Then we have
\begin{equation}
\norm{f}_{C_r(O_F^+)}\leq \frac{C_q(1-R^{2k+2})}{1-R^2}\norm{f}_{L^2(O_F^+)}.
\end{equation}
In particular, for any $f\in \text{Pol}_k(O_F^+)$, we have
\begin{equation}\label{eq31}
\norm{f}_{C_r(O_F^+)} \leq \frac{ C_q (\norm{F}^{2k+2}-1)}{\norm{F}^2-1} \norm{f}_{L^2(O_F^+)}\leq \frac{C_q\norm{F}^2}{\norm{F}^2-1}\norm{F}^{2k}\norm{f}_{L^2(O_F^+)}.
\end{equation}
\end{theorem}

\begin{proof}

Recall that $S(l,k)=\left \{|l-k|,|l-k|+2,\cdots,l+k\right\}$. Let us take an arbitrary $\xi\in \text{Pol}(O_F^+)$. Then we have
\begin{align}
&\norm{f \cdot \xi}_{L^2(O_F^+)}^2=\sum_{l=0}^{\infty}  \norm{p_l (f \cdot \xi)}_{2}^2 \leq \sum_{l=0}^{\infty}  \left (\sum_{n\in S(l,k)} \norm{p_l f p_{n}}\cdot \norm{p_{n}\xi}_2\right )^2\\
&\leq C_q^2 \norm{f}_2^2 \sum_{l=0}^{\infty} \left ( \sum_{n\in S(l,k)} R^{k+n-l} \norm{p_{n}\xi}_2\right )^2 \text{ by }\eqref{eq30.5}\\
\label{eq33}&\leq C_q^2 \norm{f}_2^2  \sum_{l=0}^{\infty} \left (  \sum_{n\in S(l,k)} R^{k+n-l}\right )\cdot \left (  \sum_{n\in S(l,k)} R^{k+n-l} \norm{p_{n}\xi}_2^2\right ).
\end{align}
Since $\displaystyle \sum_{n\in S(l,k)} R^{k+n-l}\leq R^{2k}+R^{2k-2}+\cdots +R^0 =  \frac{1-R^{2k+2}}{1-R^2}$, we have
\begin{equation}
\norm{f \cdot \xi}_{L^2(O_F^+)}^2 \leq \frac{C_q^2(1-R^{2k+2})\norm{f}_2^2}{1-R^2}\sum_{l=0}^{\infty}R^{k-l}\sum_{n\in S(l,k)} R^{n} \norm{p_{n}\xi}_2^2
\end{equation}
and $\displaystyle \sum_{l=0}^{\infty}R^{k-l}\sum_{n\in S(l,k)} R^{n} \norm{p_{n}\xi}_2^2$ is given by
\begin{align*}
\norm{p_0\xi}_2^2+&(1+R^{2})\norm{p_1\xi}_2^2+\cdots+(1+R^{2}+\cdots+R^{2k})\norm{p_k\xi}_2^2\\
&+(1+R^{2}+\cdots+R^{2k})(\norm{p_{k+1}\xi}_2^2+\norm{p_{k+2}\xi}_2^2+\cdots ).
\end{align*}
Hence we obtain
\begin{align}
\norm{f\cdot \xi}_{L^2(O_F^+)}^2\leq \frac{C_q^2(1-R^{2k+2})^2\norm{f}_2^2}{(1-R^2)^2}\norm{\xi}_2^2,
\end{align}
which tells us $\displaystyle \norm{f}_{C_r(O_F^+)}\leq \frac{C_q(1-R^{2k+2})}{1-R^2}\norm{f}_2$.
\end{proof}

Some unexpected outcomes from revisiting the arguments in \cite{BVY21} and its consequence Theorem \ref{thm1} are norm estimates for homogeneous polynomials of certain suitably chosen generators $u_{st}$. Let us start with the following lemma.

\begin{lemma}\label{lem30}
Let $F$ be a canonical matrix and suppose that both $\xi^{(k)}= e_{s_1}\otimes e_{s_2}\otimes \cdots \otimes e_{s_k}$ and $\eta^{(k)}= e_{t_1}\otimes e_{t_2}\otimes \cdots \otimes e_{t_k}$ are in $H_k$. Then $f=u_{s_1t_1}u_{s_2t_2}\cdots u_{s_k t_k}\in \text{Pol}_k(O_F^+)$ with
\begin{equation}\label{eq30.9}
\norm{u_{s_1t_1}u_{s_2t_2}\cdots u_{s_kt_k}}_{L^2(O_F^+)}=\sqrt{\frac{(F^*F)_{s_1s_1}^{-1}(F^*F)_{s_2s_2}^{-1}\cdots (F^*F)_{s_ks_k}^{-1}}{d_k}}.
\end{equation}
Moreover, the associated Fourier coefficient is given by
\begin{equation}
\widehat{f}(k)=\frac{(\xi^{(k)})^*Q(k)\xi^{(k)} }{d_k} e_{s_1 t_1}\otimes \cdots \otimes e_{s_{k-r}t_{k-r}}\otimes  e_{s_{k-r+1}t_{k-r+1}}\otimes \cdots \otimes  e_{s_{k}t_{k}},
\end{equation}
which implies
\begin{equation}
M_f(l,n)=(F^*F)_{s_{k-r+1}s_{k-r+1}}^{\frac{1}{2}}\cdots (F^*F)_{s_{k}s_{k}}^{\frac{1}{2}} (F^*F)_{t_{k-r+1}t_{k-r+1}}^{\frac{1}{2}}\cdots  (F^*F)_{t_{k}t_{k}}^{\frac{1}{2}}. 
\end{equation}

\end{lemma}

\begin{proof}
Note that
\begin{align}
u_{s_1t_1}\cdots u_{s_kt_k}&= (e_{s_1}^*\otimes \cdots \otimes e_{s_k}^*\otimes 1 ) u^{\tiny \tp \normalsize k}(e_{t_1}\otimes \cdots \otimes e_{t_k}\otimes 1)\\
&=(\xi^{(k)}\otimes 1)^* u^{(k)}(\eta^{(k)}\otimes 1)\in \text{Pol}_k(O_F^+)
\end{align}
and $(\xi^{(k)})^*Q(k)\xi^{(k)}$ is given by $(F^*F)_{s_1s_1}\cdots (F^*F)_{s_ks_k}$. Thus \eqref{eq30.9} follows from the Schur orthogonality relation, and the conclusions on $\widehat{f}(k)$ and $M_f(l,n)$ are direct from their definitions.
\end{proof}

\begin{corollary}\label{cor31}
Let $O_F^+$ be a non-Kac free orthogonal quantum group with canonical $F$-matrix, and suppose that $\xi^{(k)}_x=e_{s_1(x)}\otimes e_{s_2(x)}\otimes \cdots \otimes e_{s_k(x)}$ and $\eta^{(k)}_x=e_{t_1(x)}\otimes e_{t_2(x)}\otimes \cdots \otimes e_{t_k(x)}$ are in $H_k$ for all $x=1,2,\cdots, m$. Let 
\begin{equation} 
R= \max_{1\leq x\leq m}\max_{k-r+1\leq i\leq k} \left \{ (F^*F)_{s_i(x)s_i(x)}^{\frac{1}{4}}(F^*F)_{t_i(x) t_i(x)}^{\frac{1}{4}} \right\}.
\end{equation} 
Then, for any $\displaystyle f\in \text{span}\left \{ (\xi^{(k)}_x\otimes 1)^* u^{\tiny \tp \normalsize k}(\eta^{(k)}_x\otimes 1) : 1\leq x\leq m \right \}$, we have
\begin{enumerate}
\item $\displaystyle \norm{f}_{C_r(O_F^+)}\lesssim  R^{2k} \norm{f}_{L^2(O_F^+)}$ if $R>1$,
\item $\displaystyle \norm{f}_{C_r(O_F^+)}\approx \norm{f}_{L^2(O_F^+)}$ if $R<1$.
\end{enumerate}
Here, the constants depend only on the marix $F$.
\end{corollary}
\begin{proof}
Let us write $f_x=(\xi^{(k)}_x\otimes 1)^* u^{\tiny \tp \normalsize k}(\eta^{(k)}_x\otimes 1)$. Then 
\begin{equation}
M_f(l,n)\leq \max_{1\leq x\leq m}M_{f_x}(l,n)\leq R^{k+n-l}
\end{equation}
by Lemma \ref{lem30}, so the conclusion follows from Theorem \ref{thm1}. Note that $\gtrsim$ part in (2) is thanks to the following standard fact
\begin{equation}
 \norm{x}_{L^2(O_F^+)}\leq \norm{x}_{C_r(O_F^+)},~x\in \text{Pol}(O_F^+).
\end{equation}

\end{proof}

Corollary \ref{cor31} is general enough to cover a large class of analytic polynomials with multiple generators. Let us denote by $[a,b]=\left \{a,a+1,\cdots,b\right\}$.

\begin{example}\label{ex31}

Let $F=\displaystyle \sum_{i=1}^N \lambda_i e_{i,N+1-i}\in M_N(\mathbb{R})$ with $F^2=\text{Id}_N$,
\begin{equation}
0<|\lambda_1|\leq \cdots \leq |\lambda_n|<1<|\lambda_{N-n+1}|\leq \cdots \leq |\lambda_N|
\end{equation}
and $|\lambda_j|=1$ for all $n<j\leq N-n$. For any $1\leq i,j\leq n+1$, let us define
\begin{equation}
\mathcal{A}_{k,i,j}=\text{span}\left \{u_{s_1t_1}u_{s_2t_2}\cdots u_{s_kt_k}\right\}
\end{equation}
where $s_x\in[i,n]\cup [N-i+2,N]$ and $t_x\in [j,n]\cup [N-j+2,N]$.

\begin{enumerate}
\item For any $1\leq i,j\leq n$, we have
\begin{equation}\label{eq32}
\norm{f}_{C_r(O_F^+)}\lesssim \lambda_i^{-k}\lambda_j^{-k} \norm{f}_{L^2(O_F^+)},~f\in \mathcal{A}_{k,i,j}.
\end{equation}
\item For any $1\leq i\leq n$, we have
\begin{equation}
\norm{f}_{C_r(O_F^+)}\lesssim \lambda_i^{-k} \norm{f}_{L^2(O_F^+)},~f\in \mathcal{A}_{k,i,n+1}.
\end{equation}
\item For any $1\leq j\leq n$, we have
\begin{equation}
\norm{f}_{C_r(O_F^+)}\lesssim \lambda_j^{-k} \norm{f}_{L^2(O_F^+)},~f\in \mathcal{A}_{k,n+1,j}.
\end{equation}
\item Note that $\mathcal{A}_{k,1,1}^*=\left \{x^*: x\in \mathcal{A}_{k,1,1}\right\}$ is exactly the space $\mathcal{A}_{k,n+1,n+1}$ thanks to Lemma \ref{lem30} and \eqref{eq25}. Moreover, by Corollary \ref{cor31} and Lemma \ref{lem30}, we have
\begin{equation}
\norm{f}_{C_r(O_F^+)}=\norm{f^*}_{C_r(O_F^+)}\approx \norm{f^*}_{L^2(O_F^+)}
\end{equation}
for any homogeneous polynomial $f=\displaystyle \sum_{1\leq s_1,\cdots, t_k\leq n}x_{s_1\cdots t_k}u_{s_1t_1}\cdots u_{s_kt_k} \in \mathcal{A}_{k,1,1}$. The latter equivalence implies 
\begin{equation}\label{eq301}
\norm{f^*}_{L^p(O_F^+)}\approx  \frac{1}{\sqrt{d_k}}\left ( \sum_{1\leq s_1,t_1,\cdots,s_k,t_k\leq n} \left | x_{s_1t_1\cdots s_kt_k}\right |^2 \lambda_{t_1}^{-2}\cdots \lambda_{t_k}^{-2} \right )^{\frac{1}{2}}
\end{equation}
\normalsize for any $1\leq p\leq \infty$. Although $\norm{f}_{L^p(O_F^+)}\neq \norm{f^*}_{L^p(O_F^+)}$ in general for non-Kac $O_F^+$, a natural dual statement of \eqref{eq301} follows from \cite[Lemma 3.4 (3)]{Wa17}. Indeed we have
\begin{align}
\norm{f}_{L^p(O_F^+)}\approx \frac{1}{\sqrt{d_k}} \left ( \sum_{1\leq s_1,\cdots ,t_k\leq n} \left | x_{s_1\cdots t_k}\right |^2 \left [\lambda_{s_1}\cdots \lambda_{s_k}  \right ]^{\frac{4}{p}}\left [\lambda_{t_1}\cdots \lambda_{t_k}\right ]^{\frac{4}{p}-2} \right )^{\frac{1}{2}}
\end{align}
\normalsize for any $1\leq p\leq \infty$ and homogeneous polynomials given by $\displaystyle f= \sum_{1\leq s_1,\cdots,t_k\leq n} x_{s_1\cdots t_k}u_{s_1t_1}\cdots u_{s_k t_k}$. Here, the constants depend only on the matrix $F$.
\end{enumerate}
\end{example}

Now let us focus on the norm estimates for analytic polynomials with single generators $u_{st}$. Let $F$ be a canonical matrix discussed in Example \ref{ex31}. For any $u_{st}$ such that $F_{ss}=0=F_{tt}$ and $(F^*F)_{ss}(F^*F)_{tt}> 1$ (including $u_{st}$ with $1\leq s,t\leq n$), a direct consequence from Corollary \ref{cor31} is that
\begin{align}
\label{eq30.8} \norm{\sum_{k=0}^n x_k (u_{st})^k}_{C_r(O_F^+)}\lesssim &\sum_{k=0}^n  |x_k| \norm{(u_{st}^*)^k}_{L^2(O_F^+)}=\sum_{k=0}^n\frac{(F^*F)_{tt}^{\frac{k}{2}}}{\sqrt{d_k}}  |x_k|
\end{align}
for any $x_0,x_1,\cdots, x_n\in \Comp$. Moreover, \eqref{eq30.8} is an optimal estimate up to constant in the following sense.

\begin{theorem}\label{thm31}
Let $O_F^+$ be of non-Kac type with a canonical $F$-matrix and let $1\leq s,t\leq N$ such that $F_{ss}=0=F_{tt}$ and $(F^*F)_{ss}(F^*F)_{tt}>1$. Then we have
\begin{equation}\label{eq34}
\norm{\sum_{k=0}^n x_k (u_{st})^k}_{C_r(O_F^+)} \approx \sum_{k=0}^n  \frac{(F^*F)_{tt}^{\frac{k}{2}}}{\sqrt{d_k}}\cdot x_k 
\end{equation}
for any non-negative real numbers $x_0 ,x_1 ,\cdots, x_n\geq 0$.
\end{theorem}

\begin{proof}
Let us take the notations used in Lemma \ref{lem30} and, for simplicity, let us write $v_m=\sqrt{d_m (F^*F)_{tt}^{-m}}(u_{st}^*)^m$. Then 
\begin{equation}
(u_{st})^{m}= (\xi_m^*\otimes 1)u^{\tiny \tp \normalsize m}(\eta_m\otimes 1)\in \text{Pol}_m(O_F^+)
\end{equation} 
and $\displaystyle \left \{v_m\right\}_{m=0}^{\infty}$ is an orthonormal subset in $L^2(O_F^+)$. To get a lower bound of the operator norm of $\displaystyle \sum_{k=0}^n x_k (u_{st}^*)^k$, let us take an arbitrary sequence $(b_{m})_m\in \ell^2(\n_0)$. Then we have \small
\begin{align*}
&\left ( \sum_{k=0}^n x_k (u_{st}^*)^k\right ) \cdot \left (  \sum_{m=0}^{\infty}b_m v_m \right )=\left ( \sum_{k=0}^n x_k (u_{st}^*)^k\right ) \cdot \left (  \sum_{m=0}^{\infty}b_m \sqrt{d_m(F^*F)_{tt}^{-m}}(u_{st}^*)^m \right )\\
&=x_0b_0\sqrt{d_0 (F^*F)_{tt}^{0}} \cdot 1+ \left ( x_1b_0\sqrt{ d_0 (F^*F)_{tt}^{0}} + x_0b_1\sqrt{d_1 (F^*F)_{tt}^{-1}} \right ) u_{st}^*+\cdots\\
&+\left ( x_n b_0 \sqrt{ d_0 (F^*F)_{tt}^{0}}+x_{n-1}b_1\sqrt{d_1(F^*F)_{tt}^{-1}}+\cdots +x_0b_n\sqrt{d_n (F^*F)_{tt}^{-n}} \right ) (u_{st}^*)^n\\
&+\left ( x_n b_1 \sqrt{ d_1 (F^*F)_{tt}^{-1}}+x_{n-1}b_2\sqrt{d_2 (F^*F)_{tt}^{-2}}+\cdots +x_0b_{n+1}\sqrt{d_{n+1}(F^*F)_{tt}^{-n-1}} \right ) (u_{st}^*)^{n+1}+\cdots 
\end{align*}
\normalsize which is written as 
\begin{align*}
&(x_0b_0)v_0+ \left ( x_1b_0\sqrt{ d_0 d_1^{-1} (F^*F)_{tt} } + x_0b_1 \right ) v_1+\cdots\\
&+\left ( x_n b_0 \sqrt{d_0d_n^{-1} (F^*F)_{tt}^{n}}+x_{n-1}b_1\sqrt{d_1d_n^{-1}(F^*F)_{tt}^{n-1}}+\cdots +x_0b_n \right ) v_n\\
&+\left ( x_n b_1 \sqrt{ d_1 d_{n+1}^{-1}(F^*F)_{tt}^{n}}+x_{n-1}b_2\sqrt{d_2d_{n+1}^{-1}(F^*F)_{tt}^{n-1}}+\cdots +x_0b_{n+1} \right ) v_{n+1}+\cdots .
\end{align*}
 Using $d_m \approx r_q^m$ where $r_q=\displaystyle \frac{N_q+\sqrt{N_q^2-4}}{2}$, the $L^2$-norm of the above is equivalent to
\begin{equation}
\norm{\left (x_k \sqrt{r_q^{-k}(F^*F)_{tt}^{k}}\right )_{0\leq k\leq n}*(b_m)_{m\in \n_0}}_{\ell^2(\n_0)}
\end{equation}
up to constant where $*$ is the natural convolution product of sequences. By taking the supremum over all possible $(b_m)_{m}\in \ell^2(\n_0)$ we obtain 
\begin{equation}
\norm{ \sum_{k=0}^n x_k (u_{st}^*)^k }_{C_r(O_F^+)}\gtrsim \sum_{k=0}^{n}x_k\cdot r_q^{-\frac{k}{2}} (F^*F)_{tt}^{\frac{k}{2}}.
\end{equation}

For the upper bound, we have
\begin{align*}
\norm{\sum_{k=0}^n x_k (u_{st}^*)^k}_{C_r(O_F^+)}&\leq \sum_{k=0}^n x_k \norm{(u_{st}^*)^k}_{C_r(O_F^+)}\\
\lesssim &\sum_{k=0}^n x_k \norm{(u_{st}^*)^k}_{L^2(O_F^+)}= \sum_{k=0}^n x_k\cdot \frac{(F^*F)_{tt}^{\frac{k}{2}}}{\sqrt{d_k}} .
\end{align*}
Here the second inequality $\lesssim$ follows from Lemma \ref{lem30} and Corollary \ref{cor31}.
\end{proof}

The sharp estimate \eqref{eq34} for analytic polynomials with single generators $u_{st}$ is applicable to prove that the function $\norm{F}^{2k}$ in \eqref{eq31} is optimal up to constant. Recall that $d_k \approx r_q^k$ where $r_q=\displaystyle \frac{N_q+\sqrt{N_q^2-4}}{2}$.

\begin{corollary}\label{thm-optimal}
Let $O_F^+$ be of non-Kac type with a canonical $F$-matrix and let $\varphi:\n_0\rightarrow (0,\infty)$ be a function. Then the following are equivalent:
\begin{enumerate}
\item For any $f\in \text{Pol}_k(O_F^+)$, we have $\norm{f}_{C_r(O_F^+)}\lesssim \varphi(k)\norm{f}_{L^2(O_F^+)}$.
\item $\varphi(k)\gtrsim  \norm{F}^{2k}$ for all $k\in \n_0$.
\end{enumerate}

\end{corollary}

\begin{proof}
First of all, (2) $\Rightarrow$ (1) is clear thanks to Theorem \ref{thm1}, and let us prove the converse direction (1) $\Rightarrow$ (2). By Theorem \ref{thm31} and the assumptions, we have
\begin{align}
\lambda_1^{-k} r_q^{-\frac{k}{2}} \lesssim & \norm{(u_{11}^*)^k}_{C_r(O_F^+)} = \norm{(u_{11})^k}_{C_r(O_F^+)}\\
&\leq \varphi(k)\norm{(u_{11})^k}_{L^2(O_F^+)}\approx \varphi(k) \lambda_1^k r_q^{-\frac{k}{2}} 
\end{align}
for all $k\in \n_0$, which implies $\varphi(k)\gtrsim \lambda_1^{-2k}=\norm{F}^{2k}$.
\end{proof}

\section{Applications}

\subsection{Central $L^{2}-L^{\infty}$ multipliers}\label{sec:ultracontractivity}

A {\it central multiplier} $T_{\varphi}:\text{Pol}(O_F^+)\rightarrow \text{Pol}(O_F^+)$ is identified with a sequence $\varphi=(\varphi(k))_{k\in \n_0}$ such that
\begin{equation}
T_{\varphi}(f)=\varphi(k)f,~f\in \text{Pol}_k(O_F^+). 
\end{equation}

Recall that Corollary \ref{thm-optimal} shows that a central multiplier $T_{\varphi}$ satisfies
\begin{equation}
\norm{T_{\varphi}(f)}_{C_r(O_F^+)}\lesssim \varphi(k)\norm{f}_{L^2(O_F^+)}
\end{equation}
for all $f\in \text{Pol}_k(O_F^+)$ if and only if $\varphi(k)\gtrsim \norm{F}^{2k}$. In this section, we discuss a slightly different problem, namely characterization of $L^2$-$L^{\infty}$ boundedness of central multipliers.

\begin{theorem}\label{thm-ultra1}
Let $O_F^+$ be of non-Kac type with a canonical $F$-matrix and let $\varphi:\n_0\rightarrow (0,\infty)$ be a function. Then the following are equivalent.
\begin{enumerate}
\item The associated central multiplier $T_{\varphi}:\text{Pol}(O_F^+)\rightarrow \text{Pol}(O_F^+)$ satisfies
\begin{equation}
\norm{T_{\varphi}(f)}_{C_r(O_F^+)}\lesssim \norm{f}_{L^2(O_F^+)},~f\in \text{Pol}(O_F^+).
\end{equation}
\item $\displaystyle \sum_{k=0}^{\infty}\varphi(k)^2\norm{F}^{-4k}<\infty$.
\end{enumerate}
\end{theorem}
\begin{proof}
For (2) $\Rightarrow$ (1) part, let us repeat the arguments in \cite{FHLUZ17,BVY21} with our improved norm estimates \eqref{eq31}. For any orthogonal decomposition $f=\displaystyle \sum_{k=0}^{\infty}f_k$ with $f_k\in \text{Pol}_k(O_F^+)$, we have
\begin{align}
&\norm{\Phi_t(f)}_{C_r(O_F^+)}\leq \sum_{k=0}^{\infty}\varphi(k)\norm{f_k}_{C_r(O_F^+)}\\
&\lesssim \sum_{k=0}^{\infty}\varphi(k)\norm{F}^{2k}\norm{f_k}_{L^2(O_F^+)}\leq \left ( \sum_{k=0}^{\infty}\varphi(k)^2 \norm{F}^{4k} \right )^{\frac{1}{2}}\norm{f}_{L^2(O_F^+)}
\end{align}
by the H{\" o}lder inequality. 

For the converse direction (1) $\Rightarrow$ (2), recall that the $L^2$-norm of $\displaystyle \sum_{k=0}^n x_k (u_{11})^k$ is given by $\displaystyle \left ( \sum_{k=0}^{n} |x_k|^2 \frac{\lambda_1^{2k}}{d_k} \right )^{\frac{1}{2}}$ by Lemma \ref{lem30}. Then, combining with our assumption and Theorem \ref{thm31}, we obtain \small
\begin{align}
\left (\sum_{k=0}^{n} b_k^2 \right )^{\frac{1}{2}}&=\norm{\sum_{k=0}^{n} b_k\lambda_1^{-k}\sqrt{d_k} (u_{11})^k}_{L^2(O_F^+)}\gtrsim \norm{\sum_{k=0}^{n}  \varphi(k) b_k \lambda_1^{-k}\sqrt{d_k} (u_{11})^k}_{C_r(O_F^+)}\\
=& \norm{\sum_{k=0}^{n}  \varphi(k) b_k \lambda_1^{-k}\sqrt{d_k} (u_{11}^*)^k}_{C_r(O_F^+)}\gtrsim \sum_{k=0}^{n} \varphi(k) b_k \lambda_1^{-2k}
\end{align}
\normalsize for any non-negative real numbers $b_0,b_1,\cdots, b_n\geq 0$. Thus, by taking the supremum over all sequences $(b_k)_{k\in \n_0}$ such that $\displaystyle \sum_{k=0}^{\infty} b_k^2\leq 1$, we can conclude that $(\varphi(k) \lambda_1^{-2k})_{k\in \n_0}$ is square-summable. 
\end{proof}

An important application of Theorem \ref{thm-ultra1} is the accurate optimal time $t_F$ for ultracontractivity of the heat semigroup $(\Phi_t)_{t>0}$ satisfying that 
\begin{equation}
\norm{\Phi_t(f)}_{C_r(O_F^+)} \lesssim \norm{f}_{L^2(O_F^+)}
\end{equation} 
for all $f\in \text{Pol}(O_F^+)$ if and only if $t>t_F$. The optimal time bounds 
\begin{equation}\label{eq04}
2(N_q-2)\log\norm{F}\leq t_F\leq 2 N_q\log\norm{F}.
\end{equation}
have been proved in \cite{FHLUZ17,BVY21} where $N_q=\text{Tr}(F^*F)$. From now on, let us demonstrate how Theorem \ref{thm-ultra1} implies
\begin{equation}
t_F=2\sqrt{N_q^2-4}\log \norm{F}.
\end{equation}

The heat semigroup $(\Phi_t)_{t>0}$ consists of completely positive linear maps $\Phi_t:C_r(O_F^+)\rightarrow C_r(O_F^+)$ determined by
\begin{equation}
\Phi_t(u^{(k)}_{ij})=e^{-tc_k}u^{(k)}_{ij}
\end{equation}
for all $k\in \n_0$ and $1\leq i,j\leq n_k$. Here $c_k=\displaystyle \frac{U_k'(N_q)}{U_k(N_q)}$ and $U_k$ is the $k$-th Chebyshev polynomial  defined by $U_0(x)=1$, $U_1(x)=x$ and $xU_k(x)=U_{k+1}(x)+ U_{k-1}(x)$. An explicit formula for $U_k$ is given by
\begin{equation}
U_k(x)=\left (\frac{1}{2}\right )^{k+1} \frac{(x+\sqrt{x^2-4})^{k+1}-(x-\sqrt{x^2-4})^{k+1}}{\sqrt{x^2-4}}
\end{equation}

\begin{lemma}\label{lem40}
For any $k\in \n$ we have
\begin{equation}
\frac{U_k'(x)}{U_k(x)}=\frac{k+1}{\sqrt{x^2-4}}\cdot \left ( 1+ \frac{2(x-\sqrt{x^2-4})^{k+1}}{(x+\sqrt{x^2-4})^{k+1}-(x-\sqrt{x^2-4})^{k+1}}\right )-\frac{x}{x^2-4}.
\end{equation}
In particular, we have $\displaystyle \lim_{k\rightarrow \infty}\left \{c_k-\frac{k+1}{\sqrt{N_q^2-4}}\right\}=- \frac{N_q}{N_q^2-4}$.
\end{lemma}

Now we are ready to compute the optimal time $t_F$ for ultracontractivity of the heat semigroup $(\Phi_t)_{t>0}$ of $O_F^+$.

\begin{corollary}\label{cor:ultra2}
Let $O_F^+$ be of non-Kac type with a canonical $F$-matrix. Then the heat semigroup $(\Phi_t)_{t>0}$ satisfies 
\begin{equation}
\norm{\Phi_t(f)}_{C_r(O_F^+)}\lesssim \norm{f}_{L^2(O_F^+)}
\end{equation}
for all $f\in \text{Pol}(O_F^+)$ if and only if $t >2\sqrt{N_q^2-4}\log \norm{F}$.
\end{corollary}
\begin{proof}
It is enough to show that $\displaystyle \sum_{k=0}^{\infty}e^{-2tc_k}\norm{F}^{4k}<\infty$ holds precisely when $t>2\sqrt{N_q^2-4}\log \norm{F}$. First of all, if $t\leq 2\sqrt{N_q^2-4}\log \norm{F}$, then 
\begin{align}
&\liminf_{k\rightarrow \infty}\log \left ( e^{-2tc_k}\norm{F}^{4k}\right )=\liminf_{k\rightarrow \infty} \left (-2tc_k+4k\log\norm{F}\right )\\
&=\liminf_{k\rightarrow \infty}2(k+1)\cdot \left (-\frac{t}{\sqrt{N_q^2-4}}+2\log\norm{F}\right )\geq 0
\end{align}
which implies $\displaystyle \liminf_{k\rightarrow \infty}e^{-2tc_k}\norm{F}^{4k}\geq 1$ and $\displaystyle \sum_{k=0}^{\infty}e^{-2tc_k}\norm{F}^{4k}=\infty$.

On the other hand, if $t>2\sqrt{N_q^2-4}\log\norm{F}$, then $\displaystyle \sum_{k=0}^{\infty}e^{-2tc_k}\norm{F}^{4k}<\infty$ follows thanks to the root test
\begin{equation}
\lim_{k\rightarrow \infty}\left ( e^{-2tc_k}\norm{F}^{4k}\right)^{\frac{1}{k+1}}=e^{-\frac{2t}{\sqrt{N_q^2-4}}}\norm{F}^4<1.
\end{equation}

\end{proof}


\subsection{Comparison between the complex interpolation and the real interpolation spaces}\label{sec:interpolation}

Associated to a compatible pair of Banach spaces $(X_0,X_1)$ are interpolation spaces $(X_0,X_1)_{\theta}$ and $(X_0,X_1)_{\theta,q}$, which are defined by the {\it complex method} and the {\it real method} respectively. For a semifinite von Neumann algebra $M$ with a normal semifinite faithful trace $\varphi$, the following  interpolation spaces $L^p(M)=(M,M_*)_{\frac{1}{p}}$ and $L^{p,p}(M)=(M,M_*)_{\frac{1}{p},p}$ have equivalent norms. However, the equivalence is no longer true in the non-tracial setting as noted in \cite[Example 3.3]{PiXu03}. Indeed, the interpolation spaces $L^p(B(\ell^2))$ and $L^{p,p}(B(\ell^2))$ do not have equivalent norms with respect to $\varphi( x)=\text{Tr}(xD)$ with any diagonal injective positive operator $D\in S^1(\ell^2)$ whose trace is $1$.

We apply Theorem \ref{thm1} to establish an analogous result for free orthogonal quantum groups $O_F^+$. More precisely, in Theorem \ref{thm4}, we prove that $L^p(O_F^+)$ and $L^{p,p}(O_F^+)$ have equivalent norms for some $1<p\neq 2<\infty$ if and only if $O_F^+$ is of Kac type, i.e. the Haar state is tracial.

Our strategy is to generalize some methods developed in \cite{Yo18a} to non-Kac $O_F^+$. Let us begin with the dual statement of Theorem \ref{thm1} with modifying the proof of \cite[Proposition 3.7]{Yo18a} for non-Kac cases. We denote by $\norm{A}_{S^2_{n}}$ the Schatten $2$-norm $\text{Tr}(A^*A)^{\frac{1}{2}}$ for $A\in M_n(\Comp)$.

\begin{proposition}\label{prop51}
Let $O_F^+$ be a free orthogonal quantum group of non-Kac type. Then we have
\begin{equation}
\sup_{k\in \n_0}\left\{ \frac{\sqrt{d_k}}{\norm{F}^{2k}}\norm{\widehat{f}(k)Q(k)^{\frac{1}{2}}}_{S^2_{n_k}} \right \}\lesssim \norm{f}_{L^1(O_F^+)}
\end{equation}
for all $f\in \text{Pol}(O_F^+)$.
\end{proposition}

\begin{proof}
Recall that $\norm{f}_{L^1(O_F^+)}=\displaystyle \sup \left \{ h(x^*f): x\in \text{Pol}(O_F^+),~ \norm{x}_{C_r(O_F^+)}\leq 1\right\}$. By Theorem \ref{thm1} for any $k\in \n_0$ we have
\begin{align}
\norm{f}_{L^1(O_F^+)}\geq \sup h(x^*f) =\sup \left \{ d_k\text{Tr}(\widehat{f}(k)Q(k)\widehat{x}(k)^*)\right\}
\end{align}
where the supremum runs over all $x\in \text{Pol}_k(O_F^+)$ such that 
\begin{equation}
\frac{C_q\norm{F}^{2k+2}}{\norm{F}^2-1}\cdot \sqrt{d_k}\norm{\widehat{x}(k)Q(k)^{\frac{1}{2}}}_{S^2_{n_k}}\leq 1.
\end{equation}
Thus for any $k\in \n_0$ we obtain
\begin{equation}
\norm{f}_{L^1(O_F^+)}\geq \frac{\norm{F}^2-1}{C_q\norm{F}^{2k+2}}\cdot \sqrt{d_k}\norm{\widehat{f}(k)Q(k)^{\frac{1}{2}}}_{S^2_{n_k}},
\end{equation}
which is the desired conclusion.
\end{proof}

Now let us apply a generalized Marcinkiewicz theorem to interpolate Proposition \ref{prop51} and the Plancherel identity with adapting the proof of \cite[Theorem 3.8]{Yo18a} for non-Kac $O_F^+$.

\begin{theorem}\label{thm3}
Let $O_F^+$ be a free orthogonal quantum group of non-Kac type. Then for any $1<p\leq 2$ we have
\begin{equation}\label{eq42}
\left ( \sum_{k\geq 0}\frac{d_k^{\frac{p}{2}}}{\norm{F}^{2k(2-p)}}\norm{\widehat{f}(k)Q(k)^{\frac{1}{2}}}_{S^2_{n_k}}^p \right )^{\frac{1}{p}}\lesssim \norm{f}_{L^{p,p}(O_F^+)}
\end{equation}
for all $f\displaystyle \in \text{Pol}(O_F^+)$ whose Fourier coefficients are $\widehat{f}(k)$.
\end{theorem}

\begin{proof}

Let us define a linear map $\displaystyle T:\text{Pol}(O_F^+)\rightarrow \bigoplus_{k\geq 0}M_{n_k}(\Comp)$  given by
\begin{equation}
(Tf)(k)= \norm{F}^{2k}\sqrt{d_k}\widehat{f}(k)Q(k)^{\frac{1}{2}}.
\end{equation}
First of all, the Plancherel theorem tells us that $T$ extends to an isometry $\displaystyle T:L^2(O_F^+)\rightarrow \ell^2((S^2_{n_k})_{k\geq 0},\nu)$  with respect to  the measure $\nu(k)=\norm{F}^{-4k}$. Moreover, $T$ is of weak type (1,1) thanks to Proposition \ref{prop51}. Indeed, we have a universal constant $C>0$ such that
\begin{align}
\sum_{k: \norm{(Tf)(k)}_{S^2_{n_k}}>y} \nu(k) &\leq \sum_{k: C\norm{F}^{4k}\norm{f}_{L^1(O_F^+)}>y} \nu(k),
\end{align}
for any $y>0$. Note that the right hand side is $0$ if $f=0$. For $f\neq 0$ cases, let $k_0$ be the smallest natural number such that $C\norm{f}_{L^1(O_F^+)}>y\norm{F}^{-4k_0}$. Then the right hand side is given by 
\begin{equation}
\sum_{k=k_0}^{\infty} \norm{F}^{-4k}= \frac{\norm{F}^{-4k_0}}{1-\norm{F}^{-4}}<\frac{C\norm{f}_{L^1(O_F^+)}}{\left (1-\norm{F}^{-4}\right )y}\lesssim  \frac{\norm{f}_{L^1(O_F^+)}}{y}.
\end{equation}

Thus, applying \cite[Theorem 1.12]{BeSh88}, we obtain that
\begin{equation}
T:(L^2(O_F^+),L^1(O_F^+))_{\frac{2-p}{p},p}\rightarrow (\ell^2((S^2_{n_k})_{k\geq 0},\nu),\ell^{1,\infty}((S^2_{n_k})_{k\geq 0},\nu))_{\frac{2-p}{p},p}
\end{equation}
is bounded. Note that
\begin{equation}
(\ell^2((S^2_{n_k})_{k\geq 0},\nu),\ell^{1,\infty}((S^2_{n_k})_{k\geq 0},\nu))_{\frac{2-p}{p},p}=\ell^p((S^2_{n_k})_{k\geq 0},\nu)
\end{equation}
by  \cite[Theorem 5.3.2]{BeLo76} and we have 
\begin{align}
&(L^2(O_F^+),L^1(O_F^+))_{\frac{2-p}{p},p}\\
&=((L^{\infty}(O_F^+),L^1(O_F^+))_{\frac{1}{2}},(L^{\infty}(O_F^+),L^1(O_F^+))_1)_{\frac{2-p}{p},p}\\
&=(L^{\infty}(O_F^+),L^1(O_F^+))_{\frac{1}{p},p}=L^{p,p}(O_F^+).
\end{align}

Hence, $T:L^{p,p}(O_F^+)\rightarrow \ell^p((S^2_{n_k})_{k\geq 0},\nu)$ is bounded and the resulting inequality is \eqref{eq42}.

\end{proof}

Lastly, by applying Lemma \ref{lem30} and Theorem \ref{thm3}, we obtain the following discrimination of the interpolation spaces.

\begin{theorem}\label{thm4}
Let $O_F^+$ be a free orthogonal quantum group with a canonical $F$-matrix. Then $L^p(O_F^+)$ and $L^{p,p}(O_F^+)$ have equivalent norms for some $1<p\neq 2<\infty$ if and only if $O_F^+$ is of Kac type.
\end{theorem}
\begin{proof}
Let us focus on the only if part since the other direction is automatic in the tracial setting. If we assume that the interpolation spaces $L^{p,p}(O_F^+)=(L^{\infty}(O_F^+),L^1(O_F^+))_{\frac{1}{p},p}$ and $L^p(O_F^+)=(L^{\infty}(O_F^+),L^1(O_F^+))_{\frac{1}{p}}$ have equivalent norms for some $1<p<2$. Then we should have
\begin{equation}\label{eq51}
\left ( \sum_{k\geq 0}\frac{d_k^{\frac{p}{2}}}{\norm{F}^{2k(2-p)}}\norm{\widehat{f}(k)Q(k)^{\frac{1}{2}}}_{S^2_{n_k}}^p \right )^{\frac{1}{p}}\lesssim \norm{f}_{L^p(O_F^+)}
\end{equation}
for all $f\in L^p(O_F^+)$ by Theorem \ref{thm3}. To show that \eqref{eq51} leads us to a contraction, let us take 
\begin{equation} 
f=\sum_{k=0}^{n}x_k \sqrt{d_k}\norm{F}^{\frac{(4-p)k}{p}} (u_{11})^k\in \text{Pol}(O_F^+)
\end{equation} 
whose Fourier coefficients are given by 
\begin{equation}
\widehat{f}(k)=\displaystyle x_kd_k^{-\frac{1}{2}}\norm{F}^{\frac{(4-3p)k}{p}}\xi_k \xi_k^*\in B(H_k)\subseteq B(H_1^{\otimes k})
\end{equation} 
where $\xi_k=e_1\otimes e_1\otimes \cdots \otimes e_1\in H_k\subseteq H_1^{\otimes k}$. In this case, the left hand side of \eqref{eq51} is given by
\begin{equation}
\left ( \sum_{k=0}^n |x_k|^p\norm{F}^{-kp} \text{Tr}(\xi_k\xi_k^*Q(k))^{\frac{p}{2}} \right )^{\frac{1}{p}}=\left ( \sum_{k=0}^n |x_k|^p  \right )^{\frac{1}{p}}
\end{equation}
and the right hand side of \eqref{eq51} is estimated by 
\begin{align}
\label{eq52}&\norm{f}_{L^p(O_F^+)}=\norm{\sum_{k=0}^{n}x_k \sqrt{d_k}\norm{F}^{\frac{(4-p)k}{p}} \lambda_1^{\frac{4k}{p}}(u_{11}^*)^k}_{L^p(O_F^+)}\\
&\leq \norm{\sum_{k=0}^{n}x_k \sqrt{d_k}\norm{F}^{-k} (u_{11}^*)^k}_{L^2(O_F^+)}=\left ( \sum_{k=0}^n |x_k|^2 \right )^{\frac{1}{2}}.
\end{align}
Here the equality in \eqref{eq52} is thanks to \cite[Lemma 3.4 (c)]{Wa17}. Combining all the above arguments, we can see that the formal identity from $\ell^2(\n_0)$ into $\ell^p(\n_0)$ is a bounded map, which is a contradiction for any $1<p<2$. The conclusion for the cases $2<p<\infty$ follows from the standard duality arguments \cite[Section 1.21]{Ca64}, \cite[Theorem 3.7.1]{BeLo76}.
\end{proof}


\bibliographystyle{alpha}
\bibliography{Youn}

\end{document}